\documentclass[onefignum,onetabnum]{siamart190516}

\usepackage{lipsum}
\usepackage{amsfonts}
\usepackage{graphicx}
\usepackage{epstopdf}
\usepackage{algorithmic}
\usepackage{amsmath}
\usepackage{amsfonts}
\usepackage{multirow}
\usepackage{multirow}
\usepackage{rotating}
\usepackage{float}
\usepackage{subfigure}
\usepackage{hyperref}
\usepackage{amssymb, graphicx}%
\usepackage[applemac]{inputenc}
\usepackage{fancybox, color}

\def\bigmset{{\,\big\vert\,}}
\def\gph{\mathop{\rm gph}}
\def\epi{\mathop{\rm epi}}

\def\inte{\mathop{\rm Int }}
\def\dom{\mathop{\rm Dom}}
\def\subreg {\mathop{\rm subreg\,}}
\def\reg {\mathop{\rm reg\,}}
\def\cone{\mathop{\rm cone\,}}
\def\cl{\mathop{\rm cl\,}}
\def\N{{\mathbb{N}}}
\def\T{{\mathbb{T}}}
\def\R{{\mathbb{R}}}\def\C{{\mathcal {C}}}
\def\C{{\mathcal {C}}}

\ifpdf
  \DeclareGraphicsExtensions{.eps,.pdf,.png,.jpg}
\else
  \DeclareGraphicsExtensions{.eps}
\fi

\newsiamremark{remark}{Remark}
\newsiamremark{hypothesis}{Hypothesis}
\crefname{hypothesis}{Hypothesis}{Hypotheses}
\newsiamthm{claim}{Claim}

\headers{ Finding directional stationary points of DC programs}{H. A. Le Thi, V. N. Huynh, T. Pham Dinh }

\title{A unified approach for finding directional stationary points of DC programs}

\author{Hoai An Le Thi\thanks{Universit\'e de Lorraine, LCOMS, F-57000 Metz, France, and
Institut Universitaire de France (IUF)
  (\email{hoai-an.le-thi@univ-lorraine.fr}).}
\and Van Ngai Huynh\thanks{University of Quy Nhon, Viet Nam 
  (\email{ngaivn@yahoo.com}).}
\and Tao Pham Dinh\thanks{Laboratory of Mathematics, INSA-Rouen, University of Normandie, 76801 Saint-\'Etienne-du-Rouvray Cedex, France (\email{pham-dinh.tao@insa-rouen.fr}).}
}

\usepackage{amsopn}

\hypersetup{final}

\begin{document}

\date{}
\maketitle

\newtheorem{solution}[theorem]{Solution}
\newtheorem{summary}[theorem]{Summary}
\def\bigmset{{\,\big\vert\,}}
\def\gph{\mathop{\rm gph}}
\def\epi{\mathop{\rm epi}}
\def\inte{\mathop{\rm int }}
\def\dom{\mathop{\rm dom}}
\def\subreg {\mathop{\rm subreg\,}}
\def\reg {\mathop{\rm reg\,}}
\def\cone{\mathop{\rm cone\,}}
\def\cl{\mathop{\rm cl\,}}
\def\N{{\mathbb{N}}}
\def\T{{\mathbb{T}}}
\def\R{{\mathbb{R}}}\def\C{{\mathcal {C}}}
\def\C{{\mathcal {C}}}

\raggedbottom

\begin{abstract}
We address the problem of computing stationary points for non-smooth, non-convex optimization problems. 
While  this topic is well studied in the smooth setting, fewer algorithmic and theoretical results exist for the non-smooth case. Within Difference-of-Convex functions (DC) programming, the well-known DC Algorithm (DCA) is a standard method for computing critical points, whose definition 
depends on the chosen DC decomposition. More recently, some works have focused on computing directional stationary points - a stronger notion that does not depend on
any particular DC decomposition - for specific non-smooth DC programs, where the second DC component is the pointwise maximum of finitely many smooth convex functions.

In this contribution, we propose a new and unified approach for identifying directional stationary points of non-smooth DC programs, where both DC components
may be non-smooth. Our framework generalizes both the 
classical DCA and the above recent approaches. It applies 
when the second DC component is the pointwise maximum of a
continuous family of smooth convex functions, and more 
generally, to any continuous convex function on the whole space.
We also establish strong convergence results. Specifically, when the second DC component is the pointwise maximum of 
finitely many $C^{1,1}$ - smooth convex functions, we prove 
under the \L ojasiewicz inequality that the entire sequence
generated by our algorithm converges. This extends previous
results requiring global $C^{1,1}$ - smoothness. 
Finally, we introduce a randomized (stochastic) variant of our method and prove its almost sure convergence, thereby extending our deterministic results to the randomized setting.

\end{abstract}

\section{Introduction} 
 
The computation of stationary points for nonsmooth and nonconvex optimization problems arises in a wide range of applications in engineering, economics, machine learning, signal processing, and data science. In contrast to smooth optimization, where a rich theory and a large variety of efficient algorithms are available, the nonsmooth nonconvex setting remains particularly challenging due to the lack of differentiability and the intricate geometry of the objective landscape. These difficulties have motivated the development of suitable stationarity concepts together with dedicated algorithmic frameworks capable of effectively handling nonsmoothness and nonconvexity.

Among such frameworks, Difference of Convex functions (DC) programming and the DC Algorithm (DCA) have established themselves as a fundamental and unifying paradigm for modeling and solving nonsmooth nonconvex optimization problems (see, e.g. \cite{ATao05,LTP18,LTP23,LTP25,ANT-DCgeneral,PhLe1,PhLe2,TArecentad} and the references therein).
The standard DC optimization problem takes the form:

$$
(P) \quad \min \{ f(x) := g(x) - h(x) : x \in \mathbb{R}^n \},
$$
where $g, h: \mathbb{R}^n \to \mathbb{R} \cup \{+\infty\}$ are proper, lower semicontinuous convex functions, called DC components of $f$. DC programs involving a closed convex constraint set, namely
$$
\min \{ f(x) := g(x) - h(x) : x \in C \},
$$
with $C \subseteq \mathbb{R}^n$ closed and convex, can be rewritten as a particular  instance of $(P)$ by replacing $g$ with $g + \chi_C$, where $\chi_C$ denotes the indicator function of $C$.

Initiated in 1985 by Pham Dinh Tao \cite{Pham85}, building upon her early work on convex maximization dating back to 1974 (see the historical overview   in \cite{LTP25} and the references therein), DC programming and  DCA have undergone more than four decades of continuous development. From an initially promising theoretical framework, they have progressively evolved into a mature, powerful, and versatile methodology integrated into optimization solvers, software platforms, and real-world decision systems. Owing to their simplicity, flexibility, numerical efficiency, and scalability, DC programming and DCA have successfully addressed a broad spectrum of large-scale nonsmooth and nonconvex problems arising in engineering, machine learning, signal processing, statistics, communication networks, finance, and data science (see, e.g., \cite{ATao05,LTP18,LTP23,LTP25,ANT-DCgeneral,PhLe1,PhLe2,TArecentad} and the references therein).

Despite this remarkable success of DCA, the characterization and computation of refined stationarity notions for nonsmooth nonconvex DC optimization problems remain highly challenging.

As is well known, DCA aims to compute a \textit{critical} point of problem (P), that is, a point $x^* \in \mathbb{R}^n$ satisfying 
$$\partial g(x^*) \cap \partial h(x^*) \neq \emptyset.$$  
However, criticality generally depends on the specific DC decomposition of the objective function and is therefore not an intrinsic property of the model, since a given DC function admits infinitely many decompositions. By contrast, because every DC function is directionally differentiable in all directions, \textit{directional stationarity} emerges as a natural first-order optimality concept. It is stronger than criticality and, importantly, independent of the chosen DC decomposition.
This motivates the study of directional stationarity conditions and their computation.

A point $x^*$ is called directionally stationary if the directional derivative of $f$ at $x^*$ is nonnegative in every feasible direction (or Bouligand tangent direction in constrained settings). When $h$ is differentiable at a critical point, that point is also directionally stationary. However, in general, the two concepts differ, and a critical point of DCA may fail to satisfy the directional stationarity condition of the objective function.

Due to the difficulty of verifying directional stationarity - especially in nonsmooth and nonconvex problems such as $(P)$, where infinitely many directions must be examined - there exist relatively few algorithmic approaches in the literature for computing directionally stationary points. For instance, Beck and Hallak \cite{Beck-SIOPT} studied a class of problems in which the objective function is expressed as the difference between a smooth function and a convex one under linear constraints. More recently, Pang et al. \cite{PRA} proposed a method for computing a directionally stationary point of $(P)$ when $h$ is the 
pointwise maximum of finitely many differentiable convex functions. This approach was subsequently extended in \cite{CPS} to address statistical estimation problems, and further developed in \cite{Pang-SIOPT2018} for computing directionally stationary solutions of a class of multiblock nonsmooth optimization problems arising from generalized noncooperative potential games.
Following the works of Pang et al., the authors in \cite{LuZhouSun19} proposed several enhanced proximal DC algorithms for finding directional stationary points of DC programs in which the second DC component is defined as a possibly infinite supremum, over a specially structured compact set in a finite-dimensional Euclidean space, of a family of differentiable convex functions. In the case where the second DC component is a finite supremum, nonmonotone line-search schemes and randomized update strategies were also presented in \cite{LuZhou19}.

\vskip 0.2cm
\textbf{Our contributions.}   
In this work, we propose a new and unified algorithmic framework for computing directionally stationary points of  DC programs, extending the capabilities of DCA beyond criticality and without requiring specific structural assumptions on the DC components. Our main contributions are as follows:

i) \textit{A general algorithm for computing directional stationary points of DC programs:}
\\We design a new algorithm that integrates the notion of directional stationarity directly into the DCA framework. Unlike standard DCA, which converges to critical points, our method ensures convergence to directionally stationary (d-stationary) points. The algorithm operates without imposing structural assumptions on the DC components $g$ and $h$.
  We examine in detail several scenarios,  
  including cases where $h$ is assumed to be continuous or,
 more specifically, a convex function representing the pointwise maximum of a (possibly infinite) family of continuously differentiable convex functions. We provide insights
 into how the proposed algorithm operates in these situations. 

ii) \textit{Unification, generalization related   approaches:}
\\
Our algorithm recovers, as special cases, related works.
More precisely
\begin{description}
 \item[$\bullet$] When $h$ is differentiable this novel algorithm reduces to standard DCA. 
\item[$\bullet$] When $h$ is the pointwise maximum of a finite set of 
convex differentiable functions, our method, with a specialized update, encompasses   the algorithm of Pang et al. \cite{PRA}.
However, unlike \cite{PRA}, the convergence of our algorithm 
is established without requiring regularized quadratic terms. Moreover, our approach solves fewer convex subproblems per iteration, significantly reducing computational overhead.
\end{description}

iii)  \textit{Theoretical convergence analysis:}
\\
We provide a rigorous convergence analysis of the proposed algorithm, establishing that every accumulation point is directionally stationary under standard assumptions.
We also show that when the second component of the DC decomposition is the pointwise maximum of a finite family of $C^{1,1}$-smooth convex functions, the entire sequence generated by the algorithm converges under the \L ojasiewicz inequality. This result significantly extends the convergence analysis in \cite{ANT-JOTA18}, which required the second DC component $h$ to be globally $C^{1,1}$.

iv) \textit{Toward practical scalability - a stochastic variant for large-scale settings:}
\\
To address the potentially high computational burden associated with solving many convex subproblems per iteration, we also develop a randomized (stochastic) variant of our algorithm. By sampling components of the second DC term, this variant significantly reduces the practical complexity and scales efficiently to large-scale problems, particularly when the second DC component involves the maximum over many convex functions. Importantly, it is supported by its own rigorous analysis: we prove almost sure convergence to a directional stationary point, thereby extending our deterministic convergence results to the randomized setting.

\vskip 0.2cm
\textbf{Organization of the paper.} The remainder of this paper is organized as follows.
In Section 2, we provide essential background on Variational analysis and stationarity concepts in DC Programming, with a particular focus on directional stationarity as the targeted optimality concept.
Section 3 presents a generic algorithm for computing directional stationary points of  DC programs and discusses how it can be tailored to different scenarios.
In Section 4, we propose a randomized variant of the algorithm, designed to address large-scale problems efficiently.
Finally, Section 5 offers concluding remarks and outlines directions for future research. 

 \section{Background on Variational Analysis and Stationarity Concepts in DC Programming}
 
 In the sequel, the space $\R^n$ is equipped with the canonical inner product $\langle \cdot\rangle.$ 
 Its dual space is identified with $\R^n$ itself. 
 $\mathcal{S}(\R^n)$ denotes the set of lower semicontinuous 
 functions $\varphi:\R^n\to \R\cup\{+\infty\}.$ 
 The open and closed balls with the center $ x\in \R^n$ and radius $\varepsilon >0$ is denoted, respectively, by $B(x,\varepsilon)$ and
 $B[x,\varepsilon],$ while the unit ball (i.e., the ball 
 with the center at the origin and unit radius) is denoted by 
 $B.$ We recall some backgrounds on convex analysis and 
 nonsmooth analysis (\cite{Borisbook1}, \cite{Roc}, \cite{R-W}). \\
 A function $\varphi:\R^n\to\R\cup\{+\infty\}$ is called 
 strongly $\rho-$convex for some $\rho(\varphi)\ge 0,$  if 
 for all $x,y\in\R^n,$ $\lambda\in[0,1]$ one has
$$\varphi(\lambda x+(1-\lambda)y)\le \lambda \varphi(x)+(1-\lambda) \varphi(y)-\frac{\rho}{2}\lambda(1-\lambda)\|x-y\|^2.$$
The supremum of all $\rho\ge 0$ such that the above 
inequality is verified is called the convex modulus of $f,$ 
which is denoted by $\rho(\varphi).$\\
The (Fenchel) subdifferential of a convex function 
$\varphi\in\mathcal{S}(\R^n)$ at $x\in\dom f$ is defined by
$$\partial \varphi(x)= \{x^*\in \R^n:\quad\langle x^*,y-x\rangle\le \varphi(y)-\varphi(x)\quad\forall y\in\R^n\}.$$
We set $\partial \varphi(x)=\emptyset$ if $x\notin\dom \varphi.$
The Fenchel subdifferential of a convex function 
$\varphi$ is also defined via the \textit{directional derivative} by
$$\partial \varphi(x)= \{x^*\in \R^n:\quad\langle x^*,d\rangle\le \varphi^\prime(x,d)\quad\forall d\in\R^n\},$$
where the directional derivative $\varphi^\prime(x,d)$ is defined by
$$\varphi^\prime(x,d):=\lim_{t\to0^+}\frac{\varphi(x+td)-\varphi(x)}{t}=\inf_{t>0}\frac{\varphi(x+td)-\varphi(x)}{t}.$$

\vskip 0.2cm
Let $f:\R^n\rightarrow\R\cup\{+\infty\}$ be a lower semicontinuous real extended valued function.  The \textit{Fr\'echet subdifferential} of $f$ at $x\in\dom f$ is defined by
$$\partial^Ff(x)=\left\{x^*\in\R^n:\;\; \liminf_{h\to 0}\frac{f(x+h)-f(x)-\langle x^*,h\rangle}{\|h\|}\geq 0\right\}.$$
For $x\notin\dom f,$ we set $\partial^F f(x)=\emptyset.$ 
\vskip 0.2cm
Recall also that the Clarke-Rockafellar subdifferential of a lower semicontinuous function $f:\R^n\to\R\cup\{+\infty\}$ at $x$ is defined by
$$\partial^\uparrow f(x)=\{x^*\in\R^n:\;\; \langle x^*,d\rangle\le f^\uparrow(x,d)\;\forall d\in\R^n\}.$$
Here, $ f^\uparrow(x,d)$ is the Clarke-Rockafellar generalized directional derivative defined by
$$f^\uparrow(x,d)=\sup_{\varepsilon>0}\limsup_{y\to_f x,t\to 0^+}\inf_{v\in B(d,\varepsilon)}\frac{f(y+tv)-f(y)}{y},$$
where $y\to_f x$ means $(y,f(y))\to (x,f(x)).$
When $f$ is locally Lipschitz around $x,$ $f^\uparrow(x,d)$ coincides with the Clarke generalized derivative:
$$  f^\uparrow(x,d)=f^\circ(x,d)=\limsup_{z\to x,t\downarrow 0}\frac{f(z+td)-f(z)}{t}.$$
In this case, the Clarke-Rockafellar subdifferential (at $x$) agrees with the Clarke subdifferential, denoted by $\partial^\circ f(x).$ Morerover, when $f$ is convex,
$\partial^\circ f(x)$  coincides with the subdifferential in convex analysis.
\vskip 0.2cm
The concepts of \textit{criticality} regarded as the generalized Fermat rules are defined with respect to some differentials: A point $x\in\R^n$ is called a \textit{Fr\'{e}chet/Clarke-Rockefellar} \textit{critical} point for the function $f,$  if respectively,
\begin{equation}\label{subdiff-critical}
0\in\partial^Ff(x)/\partial^\uparrow f(x).
\end{equation}
\vskip 0.2cm
Consider a DC function $f:\R^n\to\R\cup\{+\infty\},$ $f:=g-h,$ where $g,$ $h$ is convex functions. Then one has 
\begin{equation}\label{Subdiff-DC1}
\partial^Ff(x)\subseteq\partial^\uparrow f(x) \subseteq \partial g(x)-\partial h(x);
\end{equation}
wherever  $h$ is continuous at $x\in\R^n.$ Especially, if $h$ is differentiable at $x,$ with derivative $\nabla h(x),$ then
\begin{equation}\label{Subdiff-DC2}
\partial^Ff(x)=\partial^\uparrow f(x)= \partial g(x)-\nabla h(x).
\end{equation}
In this situation, the criticality concepts based on the subdifferentials mentioned earlier all coincide.
\vskip 0.2cm
\noindent\textbf{DC strong criticality and Directional Stationarity for DC programs}
\vskip 0.2cm
For a DC function $f:\R^n\to\R\cup\{+\infty\},$ with an available DC decomposition $f:=g-h,$  recall that (\cite{ATao05}) a \textit{DC-critical} point of $f$ (with respect to the DC decomposition $g-h$) is a point $x\in\R^n$ such that 
\begin{equation}\label{DC critical}
\partial g(x)\cap \partial h(x)\not=\emptyset,\; \textrm{or equivalently,}\; 0\in \partial g(x)- \partial h(x).
\end{equation}
In the literature of DC programming, there is a clear distinction in the terminology "critical point," which generally refers to the concept of criticality described here. The DC criticality concept, as defined, relies on the specific DC decomposition in use. Based on relation (\ref{Subdiff-DC1}), the concept of DC criticality represents the least stringent form of criticality among those discussed. Additionally, in view of relation (\ref{Subdiff-DC2}), when the second component $h$ of the DC decomposition is differentiable at the given point, all forms of criticality coincide.\\
An intensified version of DC criticality is termed \textit{strong critical point:} A point $x$ is deemed a strong critical point if
\begin{equation}\label{Strong DC critical}
\emptyset\not=\partial h(x)\subseteq \partial g(x).
\end{equation}
As a DC function $f=g-h$ is amenable to directional differentiability, with 
$$f^\prime(x,d)=g^\prime(x,d)-h^\prime(x,d),\;d\in\R^n,$$ 
the concept of directional stationarity is relevant. Specifically, a point $x$ qualifies as a \textit{directional stationary} point of a DC function $f$ if
\begin{equation}
\label{direct stationary}
f^\prime (x,d)\ge 0\quad\mbox{for all}\;\; x\in\R^n.
\end{equation}
Notice that any Fr\'echet critical point of a DC function 
$f$ is necessarily a directional stationary point, although the converse does not generally hold.
For a comprehensive treatment of DC criticality, directional stationarity, and associated optimality conditions in DC programming, we refer the reader to \cite{ATao05, ANT-JOTA18}.
\vskip 0.2cm
When 
$x\in \dom(\partial h),$ the notions of strong criticality, and directional stationarity coincide, as demonstrated in the following proposition.
\begin{proposition}
\label{Coincide critical}  Given a DC function 
$f:\R^n\to\R\cup
\{+\infty\}$ with a DC decomposition $f:=g-h$, where $g$, $h$ represent lower semicontinuous convex functions. Assuming $x\in\dom (\partial h),$ the following properties are equivalent.
\begin{itemize}
\item[(i)] $x$ is a directional stationary point of $f$.
\item[(ii)] $x$ is a strong critical point  of $f$.
\end{itemize} 
In addition, if $x\in \text{int}(\dom h),$ then conditions (i) and (ii) are also equivalent to   $x$ being a Fr\'echet critical point of $f.$ 
 \end{proposition}
 \textit{Proof.} For $(i)\Rightarrow (ii),$ as $x\in \textrm{int}\left(\dom h\right),$ $\partial h(x)$ is nonempty, and since 
 $$f^\prime(x,d)=g^\prime(x,d)-h^\prime (x,d)\ge 0\quad\forall d\in\R^n,$$ for all $y\in \partial h(x),$ all $d\in\R^n,$ 
 $$\langle y,d\rangle\le h^\prime(x,d)\le g^{\prime}(x,d),$$
 therefore, $\partial h(x)\subseteq \partial g(x).$\\
 For $(ii)\Rightarrow (i),$ since $\partial h(x)\not=\emptyset,$ $h^\prime(x,d)>-\infty$ for all $d\in\R^n.$ Let $d\in\R^n$ be given. If $h^\prime(x,d)=+\infty,$ from the definition,
 $$h^\prime(x,d)=\inf_{t>0}\frac{h(x+td)-h(x)}{t},$$
 then $h(x+td)=+\infty$ for all $t>0,$ so $g(x+td)=+\infty,$ for all $t>0$ (since the set of values of the function $f$ is $(-\infty,+\infty]$). Thus  $g^\prime(x,d)=+\infty.$ Otherwise, consider now the case $h^\prime(x,d)<+\infty.$ Denote by $L:=\{td:\;t\in\R\},$ the line with the direction $d,$ and define the linear function $l$ on $L$ by $l(td):=th^\prime(x,d),\,t\in\R.$ Since $h^\prime(x,\cdot)$ is sublinear, one has $l(td)\le h^\prime(x,td),\,\forall t\in\R.$ By virtue of the Hahn-Banach theorem, there exists an linear extension of $l$ on the whole space $\R^n,$ say $x^\star\in\R^n,$ such that
 $$ \langle x^\star,v\rangle \le h^\prime(x,v),\,\forall v\in\R^n\;\;\text{and}\;\;\langle x^\star,d\rangle = h^\prime(x,d).$$
 Therefore $x^*\in\partial h(x)\subseteq\partial g(x),$ consequently,
 $\langle x^\star,d\rangle = h^\prime(x,d)\le g^\prime(x,d),$
 which completes the proof.
 \vskip 0.2cm
 \noindent When $x\in \text{int}(\dom h),$  due to the upper semicontinuous property of the subddifferential of $h$ at $x\in \textrm{int}\left(\dom h\right),$ for any $\varepsilon>0,$ there is $\delta>0$ such that
 \begin{equation}\label{Upper semicont}
 \partial h(z)\subseteq \partial h(x)+\varepsilon B,\;\;\forall z\in B(x,\delta).
 \end{equation}
 Let $\varepsilon>0$ be given. For $\delta$ along with $\varepsilon$ as above, for $z\in B(x,\delta),$ by virtue of the mean value theorem, there are $w\in [x,z]$ and $w^*\in\partial h(w)$ such that 
 $$h(z)-h(x)=\langle w^*,z-x\rangle.$$
 By (\ref{Upper semicont}), there is $y\in\partial h(x) (\subseteq \partial g(x))$ such that $\|w^*-y\|\le\varepsilon.$ This together the preceding relation imply
 $$\begin{array}{ll}
 f(z)-f(x)&=g(z)-g(x)-(h(z)-h(x))\\
 &\ge g(z)-g(x)-\langle y,z-x\rangle-\varepsilon\|z-x\|\ge -\varepsilon\|z-x\|,\;\forall z\in B(x,\delta),
 \end{array}$$
 which shows $0\in\partial^Ff(x).$
 \hfill{$\Box$}
 \section{Algorithm for finding directional stationary points of DC programs}
 We begin by recalling the standard DCA scheme for computing a critical point of a DC program.
\subsection{The standard DCA}
Consider the standard DC program $(P):$
\begin{equation}\label{P}
(P)\quad\quad\alpha:=\textrm{inf}\left\{f(x):=g(x)-h(x):\;\; x\in\R^n\right\},
\end{equation}
where $g,h:\R^n\to\R\cup\{+\infty\}$ are lower semicontinuous convex functions, and the convention $(+\infty)-(+\infty)=+\infty$ is used. Canonically, the infimum value $\alpha$ of problem $(P)$ is assumed to be finite. This assumption implies that 
\begin{equation}\label{dom relation}
\dom g\subseteq \dom h.
\end{equation}
In the convex approach to DC programming, the DC Algorithm (DCA) is based on the principles of local optimality and DC duality. It involves constructing two sequences, $\{x^k\}$ and $\{y^{k}\}$ (which are candidates for primal and dual solutions, respectively), and iteratively improving them. The sequences $\{g(x^k)-h(x^k)\}$ and $\{h^{\ast }(y^{k})-g^{\ast }(y^{k})\}$ decrease at each iteration, ultimately leading to the limits $x^{\infty }$ and $y^{\infty },$ which satisfy local optimality conditions (for more details, refer to \cite{ATao05}, \cite{PhLe1}, \cite{PhLe2}).
\vskip 0.2cm
The simplified form of the  DCA involves constructing the sequences $\{x^k\}$ and $\{y^k\}$, starting from a given point $x^0\in\dom (\partial h)$, with the following steps:
\vskip 0.2cm
\begin{equation}\notag
\boxed{\textrm{(DCA)}\quad  y^k\in \partial h(x^k); \: x^{k+1}\in\partial g^*(y^k)=\textrm{argmin}\{g(x)-\langle y^k,x\rangle: x\in\R^n\}.}
\end{equation}
It has been shown (\cite{ATao05}) that the sequence of function values  $\{f(x^k)\}$ with $\{x^k\}$ being the sequence generated by DCA  is decreasing. Furthermore, when either $g$ or $h$ is strongly convex, every limit point of $\{x^k\}$ becomes a (DC-)critical point of the optimization problem $(P).$
\vskip 0.2cm

In the following subsections, we present a generic DC algorithmic framework for solving $(P)$, which ensures convergence to a directional stationary point under appropriate assumptions, and demonstrate its effectiveness without requiring any particular structure on the function $h$. We first develop a generic scheme based on the property (proved in Remark 1) that, for any convex function $h$ and a point $x^k \in \mathrm{dom}(\partial h)$, there exists a finite family of continuously differentiable convex functions $\{h_i\}_{i\in I(x^k)}$ such that
$$
\max_{i\in I(x^k)} h_i(x^k) = h(x^k),  
\quad\text{and}\quad  
\max_{i\in I(x^k)} h_i(x) \le h(x) \quad \forall x\in \mathbb{R}^n.
$$
The generic algorithm is then specialized to three possible cases:
\\
(i) $h$ is the pointwise maximum of a finite family of continuously differentiable convex functions;\\
(ii) $h$ is the pointwise maximum of an infinite family of continuously differentiable convex functions; and\\
(iii) $h$ is an arbitrary continuous convex function.

\subsection{The generic scheme for finding a directional stationary point.}

\noindent\rule{12.7cm}{1.5pt}\\
{\bf Algorithm 1:} Finding a directional stationary point \\ 
\noindent\rule{12.7cm}{1pt}\\
\texttt{Initialization:} Initial data: $x^0\in\dom (\partial h),$ and set $k=0.$ 
 \vskip 0.2cm
\texttt{Repeat:} For $k=0,1,...,$ 
\begin{itemize}
 \item[1.] Pick a finite set of differentiable convex functions $\{h_i(x): i\in I(x^k)\}$, where $I(x^k)$ is a finite index set, such that 
 \begin{equation}\label{Diff convex minorant}
 \max_{i\in I(x^k)} h_i(x^k)=h(x^k),\;\;\mbox{and}\;\;\max_{i\in I(x^k)} h_i(x)\le h(x)\;\forall x\in\R^n.
 \end{equation}
 
 \item [2.] For each $i\in I(x^k),$ compute a solution $x^{k+1}_i$ of the convex program
 \begin{equation}\label{Convex prog}
 \min\left\{g(x)-\langle \nabla h_i(x^k),x\rangle:\;\; x\in\R^n\right\}.
 \end{equation} 

\item [3.] Set
\begin{equation}\label{next step}
x^{k+1}:=\textrm{argmin}\{g(x^{k+1}_i)-\langle \nabla h_i(x^k),x^{k+1}_i-x^k\rangle- h_i(x^k):\;\; i\in I(x^k)\}.
\end{equation}  
\item[4.] Set $k:=k+1.$\\
\texttt{Until} Stopping criterion.
\end{itemize}
\noindent\rule{12.7cm}{1.5pt}
\vskip 0.5cm
\noindent\textbf{Remark 1.} Note that the existence of a finite family of continuously differentiable convex functions satisfied (\ref{Diff convex minorant}) is equivalent to $x^k\in \dom (\partial h).$ Indeed, by setting $h_{x^k}(x)=\max_{i\in I(x^k)}h_i(x),\, x\in\R^n,$ relations (\ref{Diff convex minorant}) yield 
$$\emptyset\not=\partial h_{x^k}(x^k)\subseteq \partial h(x^k).$$ Conversely when $x^k\in\dom (\partial h),$ we can take $I(x^k):=\{y^k\}$ for any $y^k\in\partial h(x^k),$ and define the quadratic/Affine convex function:
$$h_{y^k}(x^k)=\frac{\rho(h)}{2}\|x-x^k\|^2+\langle y^k,x-x^k\rangle -h(x^k),\;\; x\in\R^n.$$
Then $h_{y^k}$ satisfies relations (\ref{Diff convex minorant}). And, in the case when $I(x^k)$ is singleton, Algorithm 1 reduces to the standard DCA.
\vskip 0.2cm
In what follows, we denote by
\begin{equation}\label{def func}
    f_i(x):=g(x)-h_i(x),\;\;\textrm{for}\; i\in I(x^k),\, k\in\N.
\end{equation}
\noindent The convergence to a directional stationary point of Algorithm 1 is stated in the following theorem.
\begin{theorem}\label{Convergence1} Suppose that
$$\alpha=\inf\{f(x)=g(x)-h(x):\;\;x\in\R^n\}>-\infty.$$
Let $\{x^k\}$ be a sequence generated by Algorithm 1. One has
\begin{itemize}
\item[(a)] For all $k\in\N,$ by setting $\rho_k:=\min_{i\in I(x^k)}\rho(h_i),$
\begin{equation}\label{Decreasing}
f(x^{k+1})\le f(x^k)-\frac{\rho_k}{2}\|x^{k+1}-x^k\|^2.
\end{equation}
As a result, $\{f(x^k)\}$ is a decreasing sequence. If $\inf_{k\in\N}\rho_k:=\rho>0$, then $\lim_{k\to\infty}\|x^{k+1}-x^k\|=0.$\\
\item[(b)] If
\begin{equation}\label{Nabla bounded}
\sup_{k\in\N}\max_{i\in I(x^k)}\|\nabla h_i(x^k)\|<+\infty,
\end{equation}
 then any limit point of $\{x^k\}$ at which $h$ is continuous is a $DC-$critical point of $(P).$ 
 \item[(c)] For given a limit point $x^\infty$ of $\{x^k\}$ at which $h$ is continuous, if for every $x\in\R^n,$ there is  a subsequence $\{x^{l_k}\}$ converging to $x^\infty$ such that  
\begin{equation}\label{D stationary cond}
\begin{array}{ll}
\limsup_{k\to\infty}\max_{i\in I(x^{l_k})}&\left[\langle\nabla h_i(x^{l_k}),x-x^{l_k}\rangle +h_i(x^{l_k})\right]\\
& \ge\langle h^\prime(x^\infty),x-x^\infty\rangle +h(x^\infty),
\end{array} 
\end{equation}
then $x^\infty$ is a directional stationary point of $(P).$
\end{itemize}
\end{theorem}
\vskip 0.2cm
\textit{Proof.} For $k\in\N,$ let $i_k\in I(x^k)$ such that $x^{k+1}=x^{k+1}_{i_k}.$ For all $i\in I(x^k),$ since $x^{k+1}$ is a solution of problem (\ref{Convex prog}),
\begin{equation}\label{estim 1}
g(x^{k+1}_i)-\langle \nabla h_i(x^k),x^{k+1}_i-x^k\rangle\le g(x)-\langle \nabla h_i(x^k),x-x^k\rangle,\;\;\forall x\in\R^n.  
\end{equation}
By the $\rho(h_{i_k})-$convexity of $h_{i_k},$
\begin{equation}\label{estim 2}
f(x^{k+1})+\frac{\rho_(h_{i_k})}{2}\|x^{k+1}-x^k\|^2\le g(x^{k+1})-\langle \nabla h_{i_k}(x^k),x^{k+1}-x^k\rangle-h_{i_k}(x^k).
\end{equation}
By the definition (\ref{next step}) of $x^{k+1},$ from relations (\ref{estim 1}) (by setting $x:=x^k$ into account) and (\ref{estim 2}), one derives 
\begin{equation}\label{estim 3}\notag
\begin{array}{ll}
f(x^{k+1})+\frac{\rho_k}{2}\|x^{k+1}-x^k\|^2&\le g(x^{k+1}_i)-\langle \nabla h_i(x^k),x^{k+1}_i-x^k\rangle-h_i(x^k)\\
&\le g(x^k)-h_i(x^k)\quad\mbox{for all}\;\; i\in I(x^k).
\end{array}
\end{equation}
Therefore, one obtains (\ref{Decreasing}): 
$$\begin{array}{ll}
f(x^{k+1})+\frac{\rho_k}{2}\|x^{k+1}-x^k\|^2&\le  g(x^k)-\max_{i\in I(x^k)}
h_i(x^k)\\
&=g(x^k)-h(x^k)=f(x^k).
\end{array}$$
Consequently, $\{f(x^k)\}$ is a decreasing sequence, and set 
\begin{equation}\label{limit value}
\lim_{k\to\infty}f(x^k)=f_\infty>-\infty.
\end{equation}
In view of relations (\ref{estim 2}), (\ref{estim 3}), 
\begin{equation}\label{estim 4}\notag
\begin{array}{ll}
f(x^{k+1})&=g(x^{k+1})-h(x^{k+1})\\
&\le g(x^{k+1})-\langle \nabla h_{i_k}(x^k),x^{k+1}-x^k\rangle-h_{i_k}(x^k)\le f(x^k).
\end{array} 
\end{equation}
 Obviously, if $\rho>0$, then (\ref{Decreasing}) yields immediately $\lim\|x^{k+1}-x^k\|=0.$\\
For $(b)$, 
 let $x^\infty$ be a limit point  of the sequence $\{x^k\}$. Let $\{x^{l_k}\}$ be a subsequence of $\{x^k\}$ converging to $x^\infty.$ As $h$ is lower semicontinuous, for each $i_{l_k}\in I(x^{l_k}),$ one has 
\begin{equation}\label{h limit}
\liminf_{k\to\infty}h_{i_{l_k}}(x^{l_k})=\liminf_{k\to\infty}h(x^{l_k})\ge h(x^\infty).
\end{equation}
  Without loss of generality, by (\ref{Nabla bounded}), one can assume $\nabla h_{i_{l_k}}(x^k)\rightarrow y^\infty\in\R^n$ as $k\rightarrow \infty.$ From
  $$\begin{array}{ll}
  \langle \nabla h_{i_{l_k}}(x^{l_k}),x-x^{l_k}\rangle&\le h_{i_{l_k}}(x)-h_{i_{l_k}}(x^{l_k})\\
  &\le h(x)-h_{i_{l_k}}(x^{l_k}),\;\;\mbox{for all}\;x\in\R^n,
  \end{array}$$
  in view of (\ref{h limit}), by letting $k\to\infty,$ one obtains
  $$\langle y^\infty, x-x^\infty\rangle \le h(x)-h(x^\infty),\;\;\mbox{for all}\;x\in\R^n.$$ 
  That is, $y^\infty\in\partial h(x^\infty).$ In the other hand, as $x^{l_k+1}$ is a solution of the problem
  $$\min\{g(x)-\langle \nabla h_{i_{l_k}}(x^{l_k}),x\rangle:\;\;x\in\R^n\},$$
  and furthermore $f$ is lower semicontinuous; $\{f(x^k)\}$ is decreasing, then $f(x^\infty)\le f(x^k)$ for all $k\in\N,$ one has
  \begin{equation}\label{optim g}
  \begin{array}{ll}
  f(x^\infty)&\le f(x^{l_k+1})=g(x^{l_k+1})-h(x^{l_k+1})\\
  &\le g(x^{l_k+1})-\left[\langle \nabla h_{i_{l_k}}(x^{l_k}),x^{l_k+1}-x^{l_k}\rangle+ h_{i_{l_k}}(x^{l_k})\right]\\
  &\le g(x)-\langle \nabla h_{i_{l_k}}(x^{l_k}),x-x^{l_k}\rangle-h_{i_{l_k}}(x^{l_k}),\;\forall x\in\R^n.
  \end{array}
  \end{equation}
  Therefore, by letting $k\rightarrow\infty,$ since $h$ is continuous at $x^\infty,$
  $$g(x^\infty)\le g(x)-\langle y^\infty, x-x^\infty\rangle,\;\forall x\in\R^n, $$
  that is, $y^\infty\in\partial g(x^\infty),$ thus
  $$y^\infty\in \partial g(x^\infty)\cap \partial h(x^\infty),$$ and so $x^\infty$ is a $DC-$critical point of $(P).$
  \vskip 0.2cm
  For $(c),$ suppose now the condition (\ref{D stationary cond}) is verified for some limit point $x^\infty$ of $\{x^k\}$. For given $x\in\R^n,$ one can find a subsequence $\{x^{l_k}\}$ converging to $x^\infty,$ satisfying (\ref{D stationary cond}). By (\ref{estim 1}), and the definition of $x^{k+1},$ for each $i\in I(x^{l_k}),$ one has
  $$\begin{array}{ll}
  f(x^\infty)&\le f(x_{l_k+1}) = g(x_{l_k+1}) -h(x_{l_k+1}) \\
  &\le g(x^{l_k+1}_i)-\langle \nabla h_{i}(x^{l_k}),x^{l_k+1}_i-x^{l_k}\rangle-h_{i}(x^{l_k})\\ 
  &\le g(x)-\langle \nabla h_i(x^{l_k}),x-x^{l_k}\rangle -h_i(x^{l_k}),\;\;\forall i\in I(x^{l_k}).
  \end{array}$$
  Consequently,
  $$\begin{array}{ll}
  f(x^\infty)
  &\le g(x)-\max_{i\in I(x^{l_k})}\left[\langle \nabla h_i(x^{l_k}),x-x^{l_k}\rangle +h_i(x^{l_k})\right].
  \end{array}$$
  By letting $k\rightarrow \infty,$  one obtains
  $$\begin{array}{ll}
  g(x^\infty)-h(x^\infty)&\le g(x)-\limsup_{k\to\infty}\max_{i\in I(x^{l_k})}\left[\langle \nabla h_i(x^{l_k}),x-x^{l_k}\rangle +h_i(x^{l_k})\right]\\
  &\le g(x)-h^\prime(x^\infty,x-x^\infty)-h(x^\infty),\end{array}$$ 
  which follows
  $$h^\prime(x^\infty,x-x^\infty)\le g(x)-g(x^\infty),\;\;\mbox{for all}\; x\in\R^n,$$ 
  equivalently, $h^\prime(x^\infty,v)\le g^\prime(x^\infty,v),$ for all $v\in\R^n,$ that is $x^\infty$ is a directional stationary point of $(P).$\hfill{$\Box$}
  \vskip 0.5cm
  \noindent\textbf{Remark 2.} 
  \begin{itemize}
  \item[2.1.] It is worth noting that Theorem \ref{Convergence1} does not require the strong convexity of either $g$ or $h$ to ensure convergence to critical or directional stationary points, unlike some existing literature on DC algorithms. This serves as a notable advantage of the presented convergence result. In contrast, in \cite{PRA}, a proximal regularization term of the form $\frac{\rho}{2}\|x-x^k\|^2$ for some $\rho>0$ was introduced. This regularization can be recognized as a specific case of the DCA applied to the DC decomposition:
  $$f=\left(g+\frac{\rho}{2}\|\cdot\|^2\right)-\left(h+\frac{\rho}{2}\|\cdot\|^2\right).$$ 
  Specifically, in Algorithm 1, for each $i\in I(x^k),$ when computing $x^{k+1}_i,$ instead of using (\ref{Convex prog}), the regularized convex programs are solved as follows:
\begin{equation}\label{Convex prog bis}
\min\left\{g(x)+\frac{\rho}{2}\|x-x^k\|^2-\langle \nabla h_i(x^k),x\rangle: x\in\R^n\right\},
\end{equation}
and then $x^{k+1}$ is updated according to
\begin{equation}\label{Next step bis}
x^{k+1}:=\textrm{argmin}\left\{f(x^{k+1}_i)+\frac{\rho}{2}\|x^{k+1}_i-x^k\|^2: \; \; i\in I(x^k)\right\}.
\end{equation}
Theorem \ref{Convergence1} remains applicable to this updated approach, with a similar proof.
\item [2.2.] Upon closer examination of the proof of Theorem \ref{Convergence1}, it becomes evident that the convergence of Algorithm 1 is still ensured if the next step $x^{k+1}$ is updated using the following approach instead of (\ref{next step}):
\begin{equation}\label{Next step @}
x^{k+1}:=\textrm{argmin}\left\{f(x^{k+1}_i): \;\; i\in I(x^k)\right\},
\end{equation}
thereby eliminating the need for a proximal regularization term as in \cite{PRA}.
\end{itemize}
\vskip 0.2cm
An important question arises: How can we update the finite family of continuously differentiable convex functions ${h_i: \;i\in I(x^k)}$ at each iteration $k$ to fulfill conditions (\ref{Diff convex minorant}) and (\ref{D stationary cond}), ensuring convergence to a directional stationary point? In the subsequent subsection, we provide explicit strategies for updating ${h_i: \;i\in I(x^k)}$ to satisfy condition (\ref{D stationary cond}), particularly when $h:\R^n\to\R$ is a continuous convex function. Notably, when $h$ is the pointwise maximum of a finite family of continuously convex functions, Algorithm 1 encompasses the method introduced in \cite{PRA}.
   
  \subsection{ The strategies to update the families of differentiable convex functions \texorpdfstring{$\{h_i:\;i\in I(x^k)\}$}{}}
  \subsubsection{Case 1: \texorpdfstring{$h$}{} is a pointwise maximum of finitely many differentiable convex functions}
Let's start by considering the case where $h$ is of the form
\begin{equation} \label{Finite maximum}
h(x):=\max\{h_i(x): \;   i=1,...,m\},
\end{equation}
where $h_i:\R^n\to\R$ are differentiable convex functions. For a given $\delta>0$ and $x\in\R^n$, we define the sets of \textit{active} and \textit{$\delta$-approximate active} indices as follows:
  \begin{equation} \label{Active set}
  M(x):=\{i\in\{1,...,m\}: \; h_i(x)=h(x)\};
  \end{equation}
  \begin{equation}\label{Approx active}
  M_\delta(x):=\{i\in\{1,...,m\}: \; h_i(x)\ge h(x)-\delta\}.
  \end{equation}
  The basic algorithm presented  in \cite{PRA} is a particular case of Algorithm 1 by taking 
\begin{equation}\label{Stra 1}
\boxed{I(x^k):= M_\delta(x^k),\; k=0,1,...,\;\mbox{for some}\;\delta>0.}
\end{equation}
By employing this updated set $I(x^k)$, it is clear that (\ref{Diff convex minorant}) is evidently satisfied. Additionally, (\ref{D stationary cond}) holds for any limit point $x^\infty$ of the sequence $\{x^k\}$. To demonstrate this, consider a subsequence $\{x^{l_k}\}$ converging to $x^\infty$. For any $\delta>0$, when $k$ is sufficiently large, it follows that $M(x^\infty)\subseteq M_\delta(x^{l_k})$. Consequently, for all $x\in\R^n$,  
\begin{equation}\label{DS Cond 1}
\begin{array}{ll}&\lim_{k\to\infty}\max_{i\in M_\delta(x^{l_k})}\left[\langle \nabla h_i(x^{l_k}),x-x^{l_k}\rangle +h_i(x^{l_k})\rangle \right]\\
&\ge\lim_{k\to\infty}\max_{i\in M(x^\infty)}\left[\langle \nabla h_i(x^{l_k}),x-x^{l_k}\rangle +h_i(x^{l_k})\rangle \right]\\
&=\max_{i\in M(x^\infty)}\langle \nabla h_i(x^\infty),x-x^\infty\rangle +h(x^\infty)\\
&=h^\prime(x^\infty,x-x^\infty) +h(x^\infty).
\end{array}
\end{equation}
We introduce an alternative update scheme that has the potential to significantly reduce the number of convex programs that need to be solved.
 \vskip 0.2cm
 \noindent\rule{12.7cm}{0.5pt}\\
 \textbf{Update 1 for $I(x^k):$} The pointwise maximum of a finite family of differentiable convex functions\\
 \noindent\rule{12.7cm}{0.5pt}
  \begin{itemize}
  \item  $k:=0,$ $x^0\in\R^n,$ $I(x^0):=\{i_0\},$ for some $i_0\in M(x^0).$
  \item Pick $\delta>0$ and an integer number $k_0.$ For $k=0,1,..,$ take some $i_{k+1}\in M(x^{k+1})$ and set
   \begin{equation}\label{Update set 1}
   I(x^{k+1}):=\left(M_\delta(x^{k+1})\setminus \bigcup_{j=k-k_0}^{k}I(x^j)\right)\bigcup\left\{i_{k+1}\right\}.
   \end{equation}
  \end{itemize}
 \noindent\rule{12.7cm}{0.5pt} 
 \vskip 0.2cm
 \noindent\textbf{Remark 3.}
 In Update 1, the index sets $I(x^k)$, which dictate the number of convex programs to be solved in step $k$, are updated in a manner that excludes index sets from some previous steps. This reduction in the index sets at each iteration leads to a decreased number of convex programs to be solved, in comparison to the approach (\ref{Stra 1}) in \cite{PRA}. This reduction in computational load is a notable advantage of our algorithm, contributing to improved efficiency and computational performance.
 \vskip 0.2cm
 Obviously, the sets $\{h_i:\;\;i\in I(x^k)\}$ satisfy relation (\ref{Diff convex minorant}) in Algorithm 1. Moreover (\ref{Convex prog}) is also fulfilled.
 \begin{lemma}\label{DS cond update 1} Suppose that $\rho(h_i)>0$ for all $i=1,...,m.$ Then the sets of differentiable convex functions 
 $$\{h_i:\;\;i\in I(x^k)\},\;\; k=0,1,..,$$  generated from Update 1 satisfy the condition (\ref{D stationary cond}) in Theorem \ref{Convergence1} for any limit point of $\{x^k\}$. 
 \end{lemma}
 \vskip 0.2cm
 \textit{Proof.} Let $x^\infty$ be a limit point of $\{x^k\}$ and let $\{x^{l_k}\}$ be a subsequence converging to $x^\infty.$ Then $M(x^\infty)\subseteq M_\delta(x^{l_k})$ when $k$ is sufficiently large. Let  $x\in\R^n$ be given arbitrarily. As
 $$\max_{i\in M(x^\infty)}\langle \nabla h_i(x^\infty), x-x^\infty\rangle =h^\prime (x^\infty,x-x^\infty),$$
 there is $\bar i\in M(x^\infty)$ with 
 $$\langle \nabla h_{\bar i}(x^\infty), x-x^\infty\rangle =h^\prime (x^\infty,x-x^\infty).$$
  In view of the update (\ref{Update set 1}) of $I(x^k)$, there is $j_k\in\{0,1,...,k_0 +1\}$ such that $\bar i\in I(x^{l_k-j_k}).$ Since $\rho(h_i)>0$ for all $i=1,...,m,$ in view of Theorem 1(b), $\lim_{k\to\infty}\|x^{k+1}-x^k\|=0,$ that implies $\lim_{k\to\infty}x^{l_k-j_k}=\lim_{k\to\infty}x^{l_k}=x^\infty.$ Hence, 
 $$\begin{array}{ll}&\lim_{k\to\infty}\max_{i\in I(x^{l_k-j_k})}\left[\langle \nabla h_i(x^{l_k-j_k}),x-x^{l_k-j_k}\rangle +h_i(x^{l_k-j_k}) \right]\\
&\ge\lim_{k\to\infty}\left[\langle \nabla h_{\bar i}(x^{l_k-j_k}),x-x^{l_k-j_k}\rangle +h_{\bar i}(x^{l_k-j_k}) \right]\\
&=\langle \nabla h_{\bar i}(x^\infty),x-x^\infty\rangle +h(x^\infty)=h^\prime(x^\infty,x-x^\infty) +h(x^\infty).
\end{array}$$ 
\hfill{$\Box$}
\vskip 0.2cm
In \cite{ANT-JOTA18}, it is demonstrated that when the function $h$ is differentiable and possesses locally Lipschitz derivatives, and the function $f$ satisfies the Lojasiewicz subgradient inequality around each limit point of $\{x^k\}$, then the entire sequence $\{x^k\}$ generated by the standard DCA converges, provided that the sequence $\{x^k\}$ remains bounded. The following theorem establishes that this convergence holds for Algorithm 1 with Update 1 as well, under the condition that $h$ is defined as the pointwise maximum of a finite family of convex functions $h_i$
of 
class $C^{1,1}$, and  the functions $f_i:=g-h_i,$ 
$j=1,...,m$ are subanalytic. 
\vskip 0.2cm
\noindent Recall that the \L ojasiewicz subgradient inequality with an exponent $\theta\in (0,1)$ is satisfied for a lower semicontiuous function $f:\R^n\rightarrow \R\cup{+\infty}$ around $\bar x\in\R^n,$ if there are $\kappa,\delta>0$ such that
 \begin{equation}\label{Loj Iqn}
 |f(x)-f(\bar x)|^\theta\le \kappa\|x^*\|,\;\;\mbox{for all}; x\in B(\bar x,\delta)\; x^*\in \partial^Ff(x).
 \end{equation} 
 The \L ojasiewicz gradient inequality was originally formulated by \L ojasiewicz \cite{Loj2} for differentiable subanalytic functions. Subsequently, its scope was expanded to encompass nonsmooth subanalytic functions in \cite{BDL}.
 
\begin{theorem}\label{Loj Conver} Consider the DC program (P) in which the function $h$ is the pointwise maximum of a finite family of differentiable convex functions $\{h_i:\;i=1,...,m\}$ with  locally Lipschitz derivatives $\nabla h_i$ ($i=1,...,m$). Assume that $\rho:=\min \{\rho(h_i):\;i=1,...,m\}>0.$ Suppose further that the sequence $\{x^k\}$ generated by Algorithm 1 with $I(x_k)\subseteq\{1,2,...,m\},$ for all $k\in\N,$ is bounded, and the functions $f_i,\;i=1,...,m,$ satisfy the \L ojasiewicz subgradient inequality around every limit points of $\{x^k\}$. Then the whole sequence $\{x^k\}$ converges to a directional stationary point of (P).
\end{theorem}
\vskip 0.2cm
\noindent\textit{Proof.} The proof is inspired by some ideas from \cite{AB},  \cite{ANT-JOTA18}. In view of Theorem \ref{Convergence1}, for $k=0,1,...,$
\begin{equation}\label{decreasing inq}
\frac{\rho}{2}\|x^{k+1}-x^k\|^2\le f(x^k)-f(x^{k+1}).
\end{equation}
and thus $\lim_{k\to\infty}\|x^{k+1}-x^k\|=0$ and $\lim_{k\to\infty}f(x^k)=f_\infty.$ Without loss of generality, we can assume that
$f_\infty=0.$ As $\{x^k\}$ is bounded and the derivatives $\nabla h_i$ ($i=1,...,m$) are locally Lipschitz, for some $L>0,$ for all $i=1,...,m,$
\begin{equation}\label{BL}
 \|\nabla h_i(x^k)\|\le L\;\;\textrm{and}\;\;\|\nabla h_i(x^k)-\nabla h_i(x^{k+1})\|\le L\|x^k-x^{k+1}\|.
 \end{equation}
For each $z\in \mathcal{C},$ let $\gamma(z),\eta(z)>0$ be such that
\begin{equation}\label{@}
\begin{array}{ll}
&|h(x)-h(z)|\le \gamma(z)/2,\;\;x\in B(z,\eta(z));\\
& f_i(z)\le f(z)+ \gamma(z)+L\eta(z),\;i\in \{1,...,m\}\;\;\Rightarrow\;\;i\in M(z).
\end{array}
 \end{equation}
For each $k=0,1,...,$ let $i_k\in I(x^k)\subseteq \{1,...,m\}$ such that 
\begin{equation}\label{!}
 x^{k+1}=x^{k+1}_{i_k}=\textrm{argmin}\{g(x)-\langle \nabla h_{i_k}(x^k), x\rangle:\;\;x\in\R^n\}.
 \end{equation}
 Denote by $\mathcal{C},$ the set of all limit points of $\{x^k\}$. Then $\mathcal{C}$ is a nonempty compact set. By assumption, for each $z\in \mathcal{C},$ there are $\theta(z)\in (0,1),$   $\kappa(z)>0,$ and $0<\delta(z)<\eta(z)$ such that for all $i=1,...,m,$
\begin{equation}\label{Loj Ina bis}
 | f_i(x)-f_i(z)|^{\theta(z)}\le \kappa(z)\| x^*\|,\;\; \forall x\in B(z,\delta(z)),\;x^*\in \partial^Ff_i(x).
 \end{equation}
 By the compactness of $\mathcal{C},$  there are a finite family of balls $\{B(z^j,\delta(z^j)):\; , z^j\in \mathcal{C},\;j=1,...,N\},$ which is a covering of $\mathcal{C}.$
 Set 
  $$\delta:=\min\{\delta(z^j):\;\;j=1,...,N\};\;\;\gamma:=\min\{\gamma(z^j):\;\;j=1,...,N\};$$
  $$\theta:=\max\{\theta(z^j):\;\;j=1,...,N\};\;\;\kappa:=\max\{\kappa(z^j):\;\;j=1,...,N\};$$
  
Since $\lim_{k\to\infty}f(x^k)=f_\infty=0$ and $\lim_{k\to\infty}\|x^{k+1}-x^k\|=0,$ there is an index $k_0$ such that
\begin{equation}\label{####}
 (0\le) f(x^k)<\gamma/2,\;\; \|x^{k+1}-x^k\|<\delta/2\;\;\textrm{and}\;\; d(x^k,\mathcal{C})<\delta/2\;\;\forall k\ge k_0-1.
 \end{equation}
Let $k\ge k_0$ be given. Pick $z^{j_k}\in \mathcal{C},$ $j_k\in\{1,...,N\}$  such that $\|x^k-z^{j_k} \|<\delta/2.$ Then, $\|x^{k+1}-z^{j_k} \|<\delta.$ Hence,
\begin{equation}\label{@@@@}
 \begin{array}{ll}
 f_{i_k}(z^{i_{k}})&\le g(z^{j_k})-\langle \nabla h_{i_k}(x^k),z^{j_k}-x^k\rangle -h_{i_k}(x^k)\\
 &\le g(z^{j_k})-g(x^{k+1})-\langle \nabla h_{i_k}(x^k),z^{j_k}-x^{k+1}\rangle \\
 &+g(x^{k+1})-\langle \nabla h_{i_k}(x^k),x^{k+1}-x^k\rangle-h_{i_k}(x^k)\\
 &\le g(z^{j_k})-g(x^{k+1})-\langle \nabla h_{i_k}(x^k),z^{j_k}-x^{k+1}\rangle \\
 &+g(x^{k+1}_i)-\langle \nabla h_{i}(x^k),x^{k+1}_i-x^k\rangle-h_{i}(x^k)\;(\forall i\in I(x^k))\\
 &\le g(z^{j_k})-g(x^{k+1})+L\delta/2 +g(x^k)-h_i(x^k)\;(\forall i\in I(x^k)).
 \end{array}
 \end{equation}
 Where, the last two inequalities follow from the definition (\ref{Convex prog}) of $x^{k+1}_i$ and the definition of $x^{k+1}$. Since $\max_{i\in I(x^k)}h_i(x^k)=h(x^k),$ and by the first relation (\ref{@}), one obtains
 \begin{equation}\label{@@@@@}\notag
 \begin{array}{ll}
 f_{i_k}(z^{i_k})&\le g(z^{j_k})-g(x^{k+1})+L\delta/2 +f(x^k)\\
 &= f(z^{j_k})+ [h(z^{j_k})-h(x^{k+1})]-f(x^{k+1})+L\delta/2 +f(x^k)\\
 &\le f(z^{j_k})+\gamma+L\delta/2\le  f(z^{j_k})+\gamma(z^{i_k})+L\eta(z^{i_k}) .
 \end{array}
 \end{equation}
 Therefore in view of the second relation (\ref{@}), one obtains $i_k\in M(z^{j_k}),$ and consequently $f_{i_k}(z^{j_k})=f(z^{j_k})=f_\infty=0.$ 
 \vskip 0.2cm
 \noindent Now let $k\ge k_0+1,$ since $\nabla h_{i_{k-1}}(x^{k-1})\in \partial g(x^k),$
 \begin{equation}
\nabla h_{i_{k-1}}(x^{k-1})-\nabla h_{i_{k-1}}(x^k)\in \partial^F f_{i_{k-1}}(x^k).
 \end{equation}
As $x^k\in B(z^{j_k},\delta),$ in view of (\ref{Loj Ina bis}), one has
\begin{equation}\label{N1}
\begin{array}{ll}
0\le f(x^k)^\theta&\le f_{i_{k-1}}(x^k)^\theta= \left(f_{i_{k-1}}(x^k)-f_{i_{k-1}}(z^{j_{k-1}})\right)^\theta\\
&\le\kappa\|\nabla h_{i_{k-1}}(x^{k-1})-\nabla h_{i_{k-1}}(x^k)\|\le\kappa L\|x^{k-1}-x^k\|. 
\end{array}
 \end{equation}
 Therefore, using the concavity of the function $t\mapsto t^{1-\theta}$ on $(0,+\infty),$ along with (\ref{decreasing inq}), we can deduce that
 \begin{equation}
\begin{array}{ll}
\frac{\rho}{2}\|x^{k+1}-x^k\|^2&\le f(x^k)-f(x^{k+1})\\
&\le (1-\theta)^{-1}(f(x^k))^{\theta}\left( f(x^k)^{1-\theta}-f(x^{k+1})^{1-\theta}\right)\\
&\le (1-\theta)^{-1}\kappa L\|x^{k-1}-x^k\|\left(f(x^k)^{1-\theta}-f(x^{k+1})^{1-\theta})\right).
\end{array}
 \end{equation}
 Then by using the inequality $a^2/b+b/4\ge a,$ for all $a,b>0,$
 \begin{equation}
\begin{array}{ll}
\|x^{k+1}-x^k\|&\le \frac{\|x^{k+1}-x^k\|^2}{\|x^{k-1}-x^k\|} +\frac{\|x^{k-1}-x^k\|}{4}\\
&\le \frac{2^{-1}\kappa L}{\rho(1-\theta)}\left(f(x^k)^{1-\theta}-f(x^{k+1})^{1-\theta}\right)+\frac{\|x^{k-1}-x^k\|}{4}.
\end{array}
 \end{equation}
 By summing the inequalities for $k=k_0+1,k_0+2,...,$ we can derive that
 \begin{equation}\label{N2}
\begin{array}{ll}
\frac{3}{4}\sum_{k=k_0+1}^\infty\|x^{k+1}-x^k\|\le \frac{2^{-1}\kappa L}{\rho(1-\theta)}f(x^{k_0+1})^{1-\theta}+\frac{\|x^{k_0}-x^{k_0+1}\|}{4},
\end{array}
 \end{equation}
 that follows the convergence of  the sequence $\{x^k\}$.\hfill$\Box$
\vskip 0.2cm
 The convergence rate, which depends on the shared \L ojasiewicz exponent $\theta$ of the functions $f\_i$, \$i = 1, 2, \dots, m\$, is given in the following theorem. 
\vskip 0.1cm
\begin{theorem}\label{CR} Suppose that the assumptions of the preceding theorem are fulfilled. Then one has
\begin{itemize}
    \item If $\theta\in (0,1/2]$ then $\|x^k-x^\infty\|=O\left(q^k\right),$ for some $q\in (0,1).$
    \item If $\theta\in (1/2,1)$ then $\|x^k-x^\infty\|=O\left(k^{\frac{\theta-1}{2\theta-1}}\right)$
\end{itemize}
    \end{theorem}
    \vskip 0.2cm
\noindent\textit{Proof.}  
The proof closely follows that of Theorem 3.5 in \cite{ANT-JOTA18}. We provide only a sketch here. By using relations (\ref{N1}) and (\ref{N2}), we establish the existence of constants $\tau_1, \tau_2 > 0$ and an index $k_0$ such that for all $k \geq k_0$, the following inequality holds:
$$
\sum_{i=k}^\infty \| x^{i+1} - x^i \| \leq \tau_1 \| x^k - x^{k-1} \| + \tau_2 \| x^k - x^{k-1} \|^{\frac{1-\theta}{\theta}}.
$$
The convergence rates stated in the theorem's conclusion are then derived from this inequality, following the same arguments as in the proof of Theorem 3.5 in \cite{ANT-JOTA18}. \hfill$\Box$

\subsubsection{\texorpdfstring{Case 2: $h$ is the pointwise maximum of infinitely many continuously differentiable convex functions}{Case 2: h is thepointwise maximum of infinitely many continuously differentiable convex functions}} 
 Consider $(P)$ when the function $h$ is defined by
 \begin{equation}\label{Continumm}
 h(x)=\max_{t\in T}\varphi(x,t),\;\; x\in\R^n.
 \end{equation}
 Where, $T$ is a compact metric space; $\varphi: \R^n\times T\to\R$ is a continuous function, and for each $t\in T$, $\varphi(\cdot,t):\R^n\to\R$ is a differentiable  convex function with derivative $\nabla_x\varphi:\R^n\times T\to\R^n$ that is continuous with respect to the joint two variables $(x,t)\in\R^n\times T.$  Denote by $h_t:=\varphi(\cdot,t),$ $t\in T$ and $\nabla_x\varphi(x,t)=\nabla h_t(x).$ As before, we use the notations for $x\in\R^n,$ $\delta>0.$
 $$M(x):=\{t\in T:\;\; h_t(x)=h(x)\}\;\;\mbox{and}\;\; M_\delta(x)=\{t\in T:\;\; h_t(x)\ge h(x)-\delta\}$$
 
 Under these given assumptions, the subdifferential and the derivative of $h$ are given by (\cite{IT})
 \begin{equation}\label{Subdiff formula}
 \partial h(x)=\textrm{co}\{\nabla h_t(x):\;\; t\in M(x)\},\; x\in\R^n; 
 \end{equation}
 \begin{equation}\label{DD formula}
 h^\prime(x,v)=\max_{t\in M(x)}\langle \nabla h_t(x),v\rangle,\;\; x\in \R^n,\; v\in\R^n.
 \end{equation}
 Where, the notation $``\textrm{co}"$ stands for the convex hull (of a set) in $\R^n.$ 
 \vskip 0.2cm
 The compactness of $T$ guarantees that for each $\varepsilon>0,$ there exist a finite family of balls with radius $\varepsilon$ in $T$ covering $T.$ So we assume that
 \vskip 0.2cm
 \noindent\textit{(A1): For each $\varepsilon>0$, there is available of a finite family of balls with radius $\varepsilon:$ $\{B_j^\varepsilon:\;\; j\in J_\varepsilon\}$ in $T$ such that $T\subseteq\bigcup_{j\in J_\varepsilon}   B_j^\varepsilon.$}
 \vskip 0.2cm
 With the assumption $(A1),$ a method to update $I(x^k)$ is given as follows.
 \vskip 0.5cm
 \noindent\rule{12.7cm}{0.5pt}\\
 \textbf{Update 2 for $I(x^k):$} The pointwise maximum of an infinite differentiable convex functions  \\
 \noindent\rule{12.7cm}{0.5pt}
  \begin{itemize}
  \item  $k:=0,$ $x^0\in\R^n,$ $I(x^0):=\{t_0\},$ for some $i_0\in M(x^0).$
  \item Pick $\delta>0$; an integer number $k_0,$ and a sequence of positive reals $\{\varepsilon_k\}\rightarrow 0.$ For $k=0,1,..,$ take some $t_{k}\in M(x^k)$ and set a finite family of balls with radius  $\varepsilon_k:$ $\{B_j^{\varepsilon_k}:\;\; j\in J_k\},$ covering $T.$  
  \item Set 
  \begin{equation}\label{J set}
  J^*_k:=\left\{j\in J_k:\;\;B_j^{\varepsilon_k}\bigcap\left(\bigcup_{p=k-k_0-1}^{k-1}I(x^p)\right)=\emptyset \right\}.
  \end{equation}
  \item Pick $t_j\in B_j^{\varepsilon_k}$ for $j\in J^*_k,$ and set 
   \begin{equation}\label{Update set 2 bis}
   I(x^k):=\{t_j:\;\; j\in J^*_k,\; t_j\in M_\delta(x^k)\}\cup\{t_k\}.
   \end{equation}
  \end{itemize}
 \noindent\rule{12.7cm}{0.5pt}
 \vskip 0.2cm 
 \begin{lemma}\label{DS cond update 2@} The sets of differentiable convex functions 
 $$\{h_{t}:\;\;t\in I(x^k)\},\;\; k=0,1,..,$$  generated from Update 2 satisfy the condition (\ref{D stationary cond}) in Theorem \ref{Convergence1} for any limit point of $\{x^k\}$, provided
 $\rho:=\inf_{t\in T}\rho(h_t)>0.$  
 \end{lemma}
 \vskip 0.2cm
 \textit{Proof.} Let $x^\infty$ be a limit point of $\{x^k\}$ and let $\{x^{l_k}\}$ be a subsequence converging to $x^\infty.$ Let  $x\in\R^n$ be given arbitrarily. Since
 $$\max_{t\in M(x^\infty)}\langle \nabla h_i(x^\infty), x-x^\infty\rangle =h^\prime (x^\infty,x-x^\infty),$$
 there is $\bar t\in M(x^\infty)$ such as 
 $$\langle \nabla h_{\bar t}(x^\infty), x-x^\infty\rangle =h^\prime (x^\infty,x-x^\infty).$$
 Since $\{B_{j}^\varepsilon:\;\; j\in J_k\}$ is a covering of $T,$ there is $j_k\in J_k$ such that $\bar t\in B_{j_k}^{\varepsilon_k}.$ Next by the continuity of $\varphi,$ when $k$ is sufficiently large, say $k\ge k_\delta,$ one has
 \begin{equation}\label{ball contain}
 B_{j_k}^{\varepsilon_k}\subseteq M_\delta(x^{l_k}),\;\;\mbox{for all}\; k\ge k_\delta.
 \end{equation}
 Since $\rho:=\inf_{t\in T}\rho(h_t)>0,$ in view of Theorem \ref{Convergence1}, 
$$\lim_{k\to\infty}x^{p_k}=\lim_{k\to\infty}x^{l_k}=x^\infty,$$ 
 for any subsequence $\{x^{p_k}\}$ with $p_k\in \{l_k-k_0-1,...,l_k\}.$ 
  In view of the update (\ref{Update set 2 bis}) for $I(x^k)$, if $j_k\notin J_{l_k}^*,$ then we can find $p_k\in \{l_k-k_0-1,...,l_k-1\}$ such that $B_{j_k}^{\varepsilon_k}\cap I(x^{p_k})\not=\emptyset,$ otherwise by (\ref{ball contain}), $B_{j_k}^{\varepsilon_k}\cap I(x^{l_k})\not=\emptyset.$ That is, for $k\ge k_\delta,$ we can find out $p_k \in\{l_k-k_0-1,...,l_k\}$ and $t_{p_k}\in B_{j_k}^{\varepsilon_k}\cap I(x^{p_k}).$ So $\lim_{k\to\infty} t_{p_k}=\bar t,$ and therefore, by the continuity of $\varphi$ and $\nabla_x\varphi,$ as well,
 $$\begin{array}{ll}&\lim_{k\to\infty}\max_{t\in I(x^{p_k})}\left[\langle \nabla h_t(x^{p_k}),x-x^{p_k}\rangle +h_t(x^{p_k}) \right]\\
&\ge\lim_{k\to\infty}\left[\langle \nabla h_{t_{p_k}}(x^{p_k}),x-x^{p_k}\rangle +h_{t_{p_k}}(x^{p_k}) \right]\\
&=\langle \nabla h_{\bar t}(x^\infty),x-x^\infty\rangle +h(x^\infty)=h^\prime(x^\infty,x-x^\infty) +h(x^\infty).
\end{array}$$ 
\hfill{$\Box$}
\vskip 0.2cm
\noindent\textbf{Remark 4.} In the update of $I(x^k),$ the number of convex programs that need to be solved during each iteration $k$ depends on the number of balls with radius $\varepsilon_k$ required to cover $T.$ Some estimates for these covering numbers of balls over specific compact sets $T$ have been established in the literature. For instance, when $T:=B_R$ is a Euclidean ball in $\R^N$ with a radius $R>0,$ an upper bound for the covering numbers $\mathcal{N}(T,\varepsilon)$ of balls with radius $\varepsilon>0$ is provided by (\cite{CS}):
 \begin{equation}\label{Covering number bound}
 \ln \mathcal{N}(B_R,\varepsilon)\le N\ln\left(\frac{4R}{\varepsilon}\right).
 \end{equation} 
 \subsubsection{Case 3: \texorpdfstring{$h:\R^n\to\R$}{} is a continuous convex function}
 Now, let's consider the scenario where $h:\R^n\to\R$ is a continuous convex function without any specific structure. It's well-known that $h$ can be represented as the supremum of a (possibly infinite) family of affine functions:
 $$h(x)=\sup_{t\in\R^n}\left(\langle t,x\rangle -h^*(t)\right)=\max_{t\in\partial h(x)}\left(\langle t,x\rangle -h^*(t)\right).$$
 For given $\delta>0,$ denote by
 \begin{equation}\label{Tdelta}
 T(x,\delta):=\bigcup_{\|z-x\|\le \delta} \partial h(z).
 \end{equation}
 The upper continuity of the subdifferential operator for a continuous convex function implies that $T(x,\delta)$ is compact for all $x\in\R^n$ and $\delta>0.$ Utilizing these sets $T(x,\delta),$ locally representing $h$ as the pointwise maximum over a compact set of affine functions is possible: for any $\delta>0,$ $x\in\R^n,$ we have
\begin{equation}\label{local represent}
 h(z)= \max_{t\in T(x,\delta)}\{h_t(z):=\varphi(z,t):= \langle t,z\rangle -h^*(t)\},\;\;\mbox{for all}\; z\in B[x,\delta].
 \end{equation}
 
 Note that when $h$ is strongly convex, its conjugate function $h^*$ is continuous throughout the entire space $\R^n,$ thus ensuring the continuity of $\varphi.$ Leveraging the (local) representation (\ref{local represent}), we can adapt the update method used for $I(x^k)$ in Case 2 to this scenario. Employing similar notations as in Case 2, for given $\delta,\varepsilon>0$ and for $x\in\R^n,$ $z\in B[x,\delta],$ we define:
$$M(z) = \{t\in T(x,\delta): h_t(z) = h(z)\};\;\;M_\varepsilon(z) = \{t\in T(x,\delta): h_t(z) \ge h(z) - \varepsilon\}.$$
Clearly, we have $$M(z) = \partial h(z)\;\;\text{and}\;\; M_\varepsilon(z) = \partial_\varepsilon h(z),\;\;\text{for all}\;z\in\R^n.$$

 \vskip 0.5cm
 \noindent\rule{12.7cm}{0.5pt}\\
 \textbf{Update 3 for $I(x^k):$} The general case of continuous convex functions  \\
 \noindent\rule{12.7cm}{0.5pt}
  \begin{itemize}
  \item  $k:=0,$ $x^0\in\R^n,$ $I(x^0):=\{t_0\},$ for some $i_0\in M(x^0).$
  \item Pick $\delta>0$; an integer number $k_0,$ and a sequence of positive reals $\{\varepsilon_k\}\rightarrow 0.$ For $k=0,1,..,$ take some $t_{k}\in M(x^k)$ and set a finite family of balls with radius  $\varepsilon_k:$ $\{B_j^{\varepsilon_k}:\;\; j\in J_k\},$ covering $T(x^k,\delta).$  
  \item Set 
  \begin{equation}\label{J set2}
  J^*_k:=\left\{j\in J_k:\;\;B_j^{\varepsilon_k}\bigcap\left(\bigcup_{p=k-k_0-1}^{k-1}I(x^p)\right)=\emptyset \right\}.
  \end{equation}
  \item Pick $t_j\in B_j^{\varepsilon_k}$ for $j\in J^*_k,$ and set 
   \begin{equation}\label{Update set 2}
   I(x^k):=\{t_j:\;\; j\in J^*_k,\; t_j\in M_\delta(x^k)\}\cup\{t_k\}.
   \end{equation}
  \end{itemize}
 \noindent\rule{12.7cm}{0.5pt}
 \vskip 0.2cm
 Since $h_t$ is not necessarily strongly convex, to ensure convergence in Theorem \ref{Convergence1}, we apply Algorithm 1 for the DC decomposition with a regularization term:
$$f=\left(g+\left(g+\frac{\rho}{2}\|\cdot\|^2\right)\right)-\left(g+\frac{\rho}{2}\|\cdot\|^2\right),\;\;\text{for}\;\rho>0,$$
which leads us to consider, at each iteration $k,$ the finite family of continuously differentiable convex functions $\left\{h_t+\frac{\rho}{2}\|\cdot\|^2:\;\;t\in I(x^k)\right\}$ instead of $\{h_t:\;\;t\in I(x^k)\}.$

Drawing upon a proof similar to that of Lemma \ref{DS cond update 2@}, we can establish:
\begin{lemma}\label{DS cond update 2} For $\rho>0,$ the sets of differentiable convex functions 
$$\left\{h_{t}+\frac{\rho}{2}\|\cdot\|^2:\;\;t\in I(x^k)\right\},\;\; k=0,1,..\;$$
generated by Update 3 satisfy the condition (\ref{D stationary cond}) in Theorem \ref{Convergence1} for any limit point of $\{x^k\}$.   
\end{lemma}

\section{The randomized method for finding a directional stationary point}
 
Regarding Algorithm 1, at each iteration $k = 0, 1, \ldots$, one must solve $|I(x^k)|$ convex optimization subproblems, where $|I(x^k)|$ is the cardinality of the index set $I(x^k)$. When $|I(x^k)|$ is large, this becomes computationally expensive. To mitigate this, we propose a stochastic variant in which only two subproblems are solved per iteration - one chosen at random and one deterministically. The proposed randomized scheme is presented below.
 
\vskip 0.5cm
\noindent\rule{12.7cm}{1.5pt}\\
{\bf Algorithm 2:} Randomized algorithm for finding a directional stationary point \\ 
\noindent\rule{12.7cm}{1pt}\\
\texttt{Initialization:} Initial data: $x^0\in\dom (\partial h),$ and set $k=0.$ 
\vskip 0.2cm
\texttt{Repeat:} For $k=0,1,...,$ 
\begin{itemize}
	\item[1.] Pick a finite set of differentiable convex functions $\{h_i(x): i\in I(x^k)\}$, where $I(x^k)$ is a finite index set, such that 
	\begin{equation}\label{Diff convex minorant2}
	\max_{i\in I(x^k)} h_i(x^k)=h(x^k),\;\;\mbox{and}\;\;\max_{i\in I(x^k)} h_i(x)\le h(x)\;\forall x\in\R^n.
	\end{equation}
	\item [2.1] Pick $j_k\in  I(x^k) $ such that $h_{j_k}(x^k)=h(x^k),$ and compute a solution $x^{k+1}_{j_k}$ of the convex program   
	\begin{equation}\label{Convex prog 1b}
	\min\left\{g(x)-\langle \nabla h_{j_k}(x^k),x\rangle:\;\; x\in\R^n\right\}.
	\end{equation} 
	\item [2.2] (\textbf{Randomized step}) Randomly select, independently from the previous iterations,  $\xi_k\in I(x^k),$ and compute a solution $x^{k+1}_{\xi_k}$ of the convex program
	\begin{equation}\label{Convex prog 2b}
	\min\left\{g(x)-\langle \nabla h_{\xi_k}(x^k),x\rangle:\;\; x\in\R^n\right\}.
	\end{equation} 
	\item [3.] Set 
	\begin{equation}\label{next step 2b}
	x^{k+1}:=\textrm{argmin}\{g(x^{k+1}_i)-\langle \nabla h_i(x^k),x^{k+1}_i-x^k\rangle- h_i(x^k):\;\; i=j_k,\xi_k\}.
	\end{equation}  
	\item[4.] Set $k:=k+1.$\\
	\texttt{Until} Stopping criterion.
\end{itemize}
\noindent\rule{12.7cm}{1.5pt}
\vskip 0.5cm
We now show  that any limit point of (P) is, \textit{almost surely},  a directional stationary point. In particular,   this recovers the corresponding result  in \cite{PRA}. 

\begin{theorem}\label{Convergence 2} Let $\{x^k\}$ be a sequence generated by Algorithm 2. Denote by $\{p_k^i:\;\;i\in I(x^k)\},$ the probability distribution of $\xi_k$ over $I(x^k).$ 
Suppose that
	\begin{itemize}
		\item[(i).] $\alpha=\inf\{f(x)=g(x)-h(x):\;\;x\in\R^n\}>-\infty.$
		\item[(ii).] Almost surely, for any limit point $x^\infty$ of $\{x^k\}$ at which $h$ is continuous, and  for every $x\in\R^n,$ there exists a subsequence $\{x^{l_k}\}$ converging to $x^\infty,$ one has 
		\begin{equation}\label{D stationary cond-bis}
		\begin{array}{ll}
		\limsup_{k\to\infty}\max_{i\in I(x^{l_k})}&\left[\langle\nabla h_i(x^{l_k}),x-x^{l_k}\rangle +h_i(x^{l_k})\right]\\
		& \ge\langle h^\prime(x^\infty),x-x^\infty\rangle +h(x^\infty).
		\end{array} 
		\end{equation}
		Then almost surely, the sequence of functional values $\{f(x^k)\}$ is a decreasing sequence and any limit point of  $\{x^k\}$ is a directional stationary point of $(P).$ Moreover, if in addition, $\rho>0$ and $1>\sup_{k\in\N}p_k\ge \inf_{k\in\N} p_k>0,$ then $\lim_{k\to\infty}\Vert x^{k+1}-x^k\Vert=0,$ almost surely.
	\end{itemize}
\end{theorem}\vskip 0.2cm
\noindent\textit{Proof.} By the definition of $x_{j_k}^{k+1}$, $x_{\xi_k}^{k+1}$ and $x^{k+1},$ 
if $x^{k+1}=x_{j_k}^{k+1},$ then 
\begin{equation}\notag
\begin{array}{ll}
f(x^{k+1})&\le g(x^{k+1})-h_{j_k}(x^{k+1})\\
&\le g(x^{k+1})-\left[\langle\nabla h_{j_k}(x^k), x^{k+1}-x^k\rangle +h_{j_k}(x^k)+\frac{\rho}{2}\|x^{k+1}-x^k\|^2\right]\\
&\le g(x^k)-h_{j_k}(x^k)-\frac{\rho}{2}\|x^{k+1}-x^k\|^2=f(x^k)-\frac{\rho}{2}\|x^{k+1}-x^k\|^2,
\end{array}
\end{equation}
Otherwise, $x^{k+1}=x_{\xi_k}^{k+1},$ one has
\begin{equation}\notag
\begin{array}{ll}
f(x^{k+1})&\le g(x^{k+1})-h_{\xi_k}(x^{k+1})\\
&\le g(x^{k+1})-\left[\langle\nabla h_{\xi_k}(x^k), x^{k+1}-x^k\rangle +h_{\xi_k}(x^k)+\frac{\rho}{2}\|x^{k+1}-x^k\|^2\right]\\
& \le g(x_{j_k}^{k+1})-\left[\langle\nabla h_{j_k}(x^k), x_{j_k}^{k+1}-x^k\rangle +h_{\xi_k}(x^k)\right]-\frac{\rho}{2}\|x^{k+1}-x^k\|^2\\
&\le g(x^k)-h_{j_k}(x^k)-\frac{\rho}{2}\|x^{k+1}-x^k\|^2=f(x^k)-\frac{\rho}{2}\|x^{k+1}-x^k\|^2,\end{array}
\end{equation}
Therefore, one always has
\begin{equation}\notag
f(x^{k+1})\le f(x^k)-\frac{\rho}{2}\|x^{k+1}-x^k\|^2,\;\;\forall k\in\N.
\end{equation}
Thus
\begin{equation}\label{!!!}
\begin{array}{ll}
\mathbb{E}\left[f(x^{k+1})|x^0,x^1,\cdots x^k\right]&\le f(x^k)-\frac{\rho}{2}\mathbb{E}\left[\|x^{k+1}-x^k\|^2|x^0,x^1,\cdots x^k\right]
\end{array}
\end{equation}
By the supermartingale convergence theorem, $\{f(x^k)\}$ is decreasing and converges almost surely, say, to $f_\infty\ge\alpha,$ and if $\rho>0,$ then
 $$\sum_{k=1}^\infty\mathbb{E}\left[\|x^{k+1}-x^k\|^2\vert x^0,x^1,\cdots x^k\right]<+\infty.
$$
The last relation implies $\lim_{k\to\infty}\Vert x^{k+1}-x^k\Vert=0,$ almost surely. 
\vskip 0.2cm
On the other hand, for any subsequence $\{x^{l_k}\},$ for all $x\in\R^n,$ one has (almost surely), if $x^{l_k+1}=x_{\xi_k}^{l_k+1},$ then
\begin{equation}\notag
\begin{array}{ll} f_\infty\le f(x^{l_k+1})&\le  g(x^{l_k+1}) -h_{\xi_k}(x^{l_k+1}) \\
&\le  g(x_{\xi_{l_k}}^{l_k+1})-\left[\langle \nabla h_{\xi_{l_k}}(x^{l_k}), x_{\xi_{l_k}}^{l_k+1}-x^k\rangle +h_{\xi_{l_k}}(x^{l_k})\right]\\
&\le g(x)-\left[\langle\nabla h_{\xi_{l_k}}(x^{l_k}), x-x^{l_k}\rangle +h_{\xi_{l_k}}(x^{l_k})\right],\end{array}
\end{equation}
otherwise, $x^{l_k+1}=x_{j_k}^{l_k+1},$ then
\begin{equation}\notag
\begin{array}{ll} f_\infty\le f(x^{l_k+1})&\le  g(x^{l_k+1}) -h_{j_k}(x^{l_k+1}) \\
&\le  g(x_{j_{l_k}}^{l_k+1})-\left[\langle\nabla h_{j_{l_k}}(x^{l_k}), x_{j_{l_k}}^{l_k+1}-x^k\rangle +h_{j_{l_k}}(x^{l_k})\right]\\
&\le  g(x_{\xi_{l_k}}^{l_k+1})-\left[\langle\nabla h_{\xi_{l_k}}(x^{l_k}), x_{\xi_{l_k}}^{l_k+1}-x^k\rangle +h_{\xi_{l_k}}(x^{l_k})\right]\\
&\le g(x)-\left[\langle\nabla h_{\xi_{l_k}}(x^{l_k}), x-x^{l_k}\rangle +h_{\xi_{l_k}}(x^{l_k})\right].\end{array}
\end{equation}
That is, in the two cases, one always has
\begin{equation}\label{eqn bis}
f_\infty\le g(x)-\left[\langle\nabla h_{\xi_{l_k}}(x^{l_k}), x-x^{l_k}\rangle +h_{\xi_{l_k}}(x^{l_k})\right]. 
\end{equation}
For each fixed index $l_k,$ given $x,$ denote 
$$J_k:=\{j\in I(x^{l_k}):\;\;j=\text{argmax}_{i\in I(x^{l_k})}\left[\langle\nabla h_i(x^{l_k}), x-x^{l_k}\rangle +h_{i}(x^{l_k})\right]\},$$
and set $p_k=\mathbb{P}(\xi_{l_k}\in J_k)>0.$ Consider the random variable 
$$W=g(x)-\left[\langle\nabla h_{\xi_{l_k}}(x^{l_k}), x-x^{l_k}\rangle +h_{\xi_{l_k}}(x^{l_k})\right],\,\text{if}\,\, \xi_k\in J_k,\,,\text{otherwise},\ f_\infty.$$
Relation (\ref{eqn bis}) yields immediately
$$f_\infty\le \mathbb{E}W= (1-p_k)f_\infty+ p_kg(x)-p_k\max_{i\in I(x^{l_k})}\left[\langle\nabla h_i(x^{l_k}), x-x^{l_k}\rangle +h_{i}(x^{l_k})\right],$$
or equivalently,
\begin{equation}
\label{bound}
f_\infty\le 
g(x)-\max_{i\in I(x^{l_k})}\left[\langle\nabla h_i(x^{l_k}), x-x^{l_k}\rangle +h_{i}(x^{l_k})\right].
\end{equation}
Let us now consider a realization $\{x^k\}$ of the set of probability one realizations for which $\lim_{k\to\infty}f(x^k)=f_\infty$ and relation (\ref{D stationary cond-bis}) in condition (ii) holds. 
For $x^\infty$ being a limit point of $\{x^k\}$ at which $h$ is continuous, for given $x\in\R^n,$ pick a  subsequence $\{x^{l_k}\}$ converging to $x^\infty,$ verifying (\ref{D stationary cond-bis}). Then $f(x^\infty)=f_\infty,$ by letting $k\to\infty$ in relation (\ref{bound}), in view of (\ref{D stationary cond-bis}), one derive that
$$f(x^\infty)=g(x^\infty)-h(x^\infty)\le g(x)-h^\prime (x^\infty, x-x^\infty)-h(x^\infty),$$
or
\begin{equation}\notag
g(x^\infty)\le 
g(x)-h^\prime (x^\infty, x-x^\infty),\;\forall x\in\R^n,
\end{equation}
which follows that $x^\infty$ is a directional stationary point of (P).\hfill{$\Box$}
\vskip 0.5cm
\noindent \textbf{Remark 5.} Consider again problem (P) in which the function $h$ is of the forms (\ref{Finite maximum}) and (\ref{Continumm}) in the previous section. For the case (\ref{Finite maximum}), obviously, when $I(x^k)$ is the whole set of indices $\{1,2,\cdots,m\}$ or the approximation active set $M_\delta(x^k)$ for $\delta>0$, then assumption (ii) of Theorem \ref{Convergence 2} is well satisfied. For Updates 1 and 2 of $I(x^k),$ in view of Lemmas \ref{DS cond update 1} and \ref{DS cond update 2}, this assumption is satisfied provided the minimum of strongly convex constants of all component functions $h_i's$, denote by $\rho$ is positive. 
\vskip 0.5cm
We consider problem (P) in which the function $h$ is of the form (\ref {Continumm}):
\begin{equation}\notag
h(x)=\max_{t\in T} h_t(x):=\varphi(x,t),\;\; x\in\R^n,
\end{equation}
with the assumptions and the notations as in Subsection 3.3.3. Suppose that the metric compact space $T$ is endowed with a Borel probability mesure $\mathbb{P}$ such that the probability of any balls in $T$ is positive. Next we present a randomized algorithm along with the propability distibution $\mathbb{P}$ for finding a directional stationary point of (P) when $h$ is given by (\ref {Continumm}).
\vskip 0.5cm
\noindent\rule{12.7cm}{1.5pt}\\
{\bf Algorithm 3:} Randomized algorithm for $h$ of the form (\ref {Continumm})\\ 
\noindent\rule{12.7cm}{1pt}\\
\texttt{Initialization:} Initial data: $x^0\in\dom (\partial h),$ and set $k=0.$ 
\vskip 0.2cm
\texttt{Repeat:} For $k=0,1,...,$ 
\begin{itemize}
	\item [1.] Pick $t_k\in M(x^k) $  and compute a solution $x^{k+1}_{t_k}$ of the convex program   
	\begin{equation}\label{Convex prog 1}
	\min\left\{g(x)-\langle \nabla h_{t_k}(x^k),x\rangle:\;\; x\in\R^n\right\}.
	\end{equation} 
	\item [2.] (\textbf{Randomized step}) Randomly select  $\xi_k\in T$ with $\xi_k\simeq \mathbb{P}$ independently from the previous iterations, and compute a solution $x^{k+1}_{\xi_k}$ of the convex program
	\begin{equation}\label{Convex prog 2}
	\min\left\{g(x)-\langle \nabla h_{\xi_k}(x^k),x\rangle:\;\; x\in\R^n\right\}.
	\end{equation} 
	\item [3.] Set
	\begin{equation}\label{next step 2}
	x^{k+1}:=\textrm{argmin}\{g(x^{k+1}_i)-\langle \nabla h_t(x^k),x^{k+1}_t-x^k\rangle- h_t(x^k):\;\; t=t_k,\xi_k\}.
	\end{equation}  
	\item[4.] Set $k:=k+1.$\\
	\texttt{Until} Stopping criterion.
\end{itemize}
\noindent\rule{12.7cm}{1.5pt}
\vskip 0.5cm
The almost surely convergence to a directional stationary point is shown in the following theorem.
\begin{theorem}\label{Convergence 3} Consider (P) with the function $h$ of the form (\ref {Continumm}). Let $\{x^k\}$ be a sequence generated by Algorithm 3. With the stated assumptions in Subsection 3.3.3. Assume that almost surely, the sequence $\{x^k\}$ is bounded.
	Then $\{f(x^k)\}$ is a decreasing sequence and almost surely any limit point of  $\{x^k\}$ is  a directional stationary point of $(P).$
\end{theorem}
\vskip 0.2cm
\noindent \textit{Proof.} Similarly to the proof of Theorem \ref{Convergence 2}, the sequence of functional values $\{f(x^k)\}$ is decreasing and converges almost surely to $f_\infty$. 
Then, as in the argument to obtain (\ref{eqn bis}) in the proof of the preceding theorem, one has
\begin{equation}\label{@+}
f_\infty\le g(x)-\left[\langle\nabla h_{\xi_{l_k}}(x^{k}), x-x^{k}\rangle +h_{\xi_{k}}(x^{k})\right],\;\;\forall k\in \N,\,\forall x\in\R^n.
\end{equation}
Given $x\in\R^n,$ given $\delta>0,$ for each $k\in\N,$ pick 
$$t_k:=t_k(x)=\text{argmax}_{t\in M(x^k,\delta)}\left[\langle\nabla h_t(x^{k}), x-x^{k}\rangle +h_{t}(x^{k})\right]$$
Pick any $\varepsilon>0,$ since by assumption, both $\varphi,$ $\nabla_x \varphi(\cdot,\cdot)$ are continuous in the joint variable $(x,t),$ we can find $ \delta_k>0,$ 
\begin{equation}\label{UC}
|h_t(x^k)-h_{t_k}(x^k)|<\varepsilon,\;\; \|\nabla h_t(x^k)-\nabla h_{t_k}(x^k)\|<\varepsilon,\;\forall k\in\N,\; \forall t\in B(t_k,\delta_k).
\end{equation}
For each $k\in\N,$ consider the random variable:
\[
W_k =
\left\{
\begin{array}{ll}
g(x) - \big[\langle \nabla h_{t_k}(x^{k}), x-x^k\rangle + h_{t_k}(x^k)\big]
+ \varepsilon \|x-x^k\| + \varepsilon,
& \text{if } \xi_k \in B(t_k,\delta_k),\\
f_\infty, & \text{otherwise.}
\end{array}
\right.
\]
Then relations (\ref{@}) and (\ref{UC}) yield $f_\infty\le W_k,$
so for $p_k=\mathbb{P}\left(B(t_k,\delta_k\right)>0,$
\begin{align*}
f_\infty \le \mathbb{E}W_k
&= (1-p_k)f_\infty
+ p_k g(x)
- p_k \max_{t\in M(x^k,\delta)}
\left[\langle \nabla h_t(x^{k}), x-x^{k}\rangle + h_{t}(x^{k})\right] \\
&\quad + p_k\left(\varepsilon\|x-x^k\|+\varepsilon\right).
\end{align*}
or equivalently,
\begin{equation}
\label{bound bis}
f_\infty\le 
g(x)-\max_{t\in M(x^k,\delta)}\left[\langle\nabla h_t(x^{k}), x-x^{k}\rangle +h_{t}(x^{k})\right]+ \varepsilon\|x-x^k\|+\varepsilon\;\forall x\in\R^n.
\end{equation}
Now consider a realization $\{x^k\}$ of the set of probability one realizations for which $\lim_{k\to\infty}f(x^k)=f_\infty.$ 
For $x^\infty$ being a limit point of $\{x^k\}.$ Without loss of generality, assume the whole sequence $\{x^k\}$  converging to $x^\infty.$ Then $f(x^\infty)=f_\infty,$ by letting $k\to\infty$ in relation (\ref{bound bis}), as $M(x^\infty)\subseteq M(x^k,\delta)$ when $k$ is sufficiently large, one arrive at
$$f(x^\infty)=g(x^\infty)-h(x^\infty)\le g(x)-h^\prime (x^\infty, x-x^\infty)-h(x^\infty)+\varepsilon\|x-x^\infty\|+\varepsilon,$$
as $\varepsilon>0$ arbitrary, one has
\begin{equation}\notag
g(x^\infty)\le 
g(x)-h^\prime (x^\infty, x-x^\infty),\;\forall x\in\R^n,
\end{equation}
showing that $x^\infty$ is a directional stationary point of (P).\hfill{$\Box$}
\section{Conclusion}

This paper makes a significant contribution to DC programming by introducing a powerful, unified framework for 
computing directional stationary points in general (smooth / non-smooth) DC programs, supported by rigorous convergence 
guarantees. By leveraging a tailored update strategy, our 
method addresses the general non-smooth DC setting - substantially extending and improving upon the particular 
case considered by Pang et al. \cite{PRA} - while also 
reducing the number of convex programs that must be solved
at each iteration.
The framework is especially effective when the second DC 
component is the pointwise maximum of a finite or compact 
family of smooth convex functions, but it also extends 
naturally to general DC programs.

A randomized variant further enhances scalability by significantly lowering the computational burden in large-scale settings, while retaining rigorous convergence guarantees. Together, these advances establish a practical and theoretically sound foundation for tackling challenging non-smooth DC problems.

Future research will explore broader problem settings and applications where the proposed algorithm can be effectively deployed, while validating its practical impact through large-scale, diverse numerical experiments - paving the way for its adoption in real-world optimization challenges.

\end{document}